\documentclass[12pt]{amsart}
\usepackage{fullpage}

\usepackage{url,amssymb,amsmath,mathrsfs}
\usepackage[alphabetic,backrefs,lite]{amsrefs}

\newtheorem{theorem}{Theorem}
\newtheorem{lemma}[theorem]{Lemma}
\theoremstyle{definition}
\newtheorem{definition}[theorem]{Definition}

\usepackage[backref]{hyperref}
\hypersetup{
    colorlinks=true,       % false: boxed links; true: colored links
    linkcolor=magenta,          % color of internal links
    citecolor=blue,        % color of links to bibliography
    filecolor=blue,      % color of file links
    urlcolor=red           % color of external links
}

\newcommand{\adj}{\operatorname{adj}}
\newcommand{\PM}{\mathbb{P}\mathscr{M}}

\begin{document}

\title{$3 \times 3$~singular matrices of linear forms}
\author[]{Damiano Testa}
\begin{abstract}
We determine the irreducible components of the space of $3 \times 3$~matrices of linear forms with vanishing determinant.  We show that there are four irreducible components and we identify them concretely.  In particular, under elementary row and column operations with constant coefficients, a $3 \times 3$~matrix with vanishing determinant is equivalent to one of the following four: a matrix with a zero row, a zero column, a zero $2 \times 2$~square or an antisymmetric matrix.
\end{abstract}
\address{\vfill\begin{minipage}{250pt}
Mathematical Institute \\
Zeeman Building \\
University of Warwick \\
Coventry CV4 7AL \\
UK\vphantom{${W_W}$}
\end{minipage}}
\date{\today}
\subjclass[2010]{14M12, 15A54, 13C40}
\keywords{Determinantal varieties, determinantal ideals, matrices of linear forms}
\email{\begin{minipage}[t]{10pt}
adomani@gmail.com
\end{minipage}}
\maketitle

\section*{Introduction}

If $X$ is a square matrix over a field and the determinant of $X$ vanishes, it follows that there is a non-zero linear combination of the rows of $X$ that vanishes.  This is no longer true if the matrix is a matrix of linear forms and we only allow linear combinations by constants.  For instance, letting $x_1, \ldots , x_5$ be independent variables, no non-trivial constant linear combination of the rows of the matrices
\begin{equation} \label{eq:ese}
\begin{pmatrix}
x_1 & x_1 \cr
x_2 & x_2
\end{pmatrix}
\quad \quad
\begin{pmatrix}
x_1 & x_2 & x_3 \cr
x_4 & 0 & 0 \cr
x_5 & 0 & 0
\end{pmatrix}
\end{equation}
vanishes, although both matrices have zero determinant.  In the case of the first matrix in~\eqref{eq:ese}, the difference of the {\emph{columns}} vanishes, while in the case of the second matrix there is also no non-trivial linear combination of the columns that vanishes.  Thus, we have identified three types of $3 \times 3$~matrices with vanishing determinants: matrices with a zero row, matrices with a zero column and matrices with a zero $2 \times 2$~square.  Up to multiplication on the left and on the right by an invertible $3 \times 3$~matrix of constants, there is one more kind of matrix that has vanishing determinant (Theorem~\ref{t:comp}) and is not of one of the three types above: antisymmetric matrices (see Table~\ref{ta:mat}).

\begin{table}[h]
$\begin{array}{c@{\hspace{20pt}}c@{\hspace{20pt}}c@{\hspace{20pt}}c}
\begin{pmatrix}
\star & \star & \star \cr
\star & \star & \star \cr
0 & 0 & 0
\end{pmatrix}
&
\begin{pmatrix}
\star & \star & \star \cr
0 & \star & 0 \cr
0 & \star & 0
\end{pmatrix}
&
\begin{pmatrix}
0 & x_1 & -x_2 \cr
-x_1 & 0 & x_3 \cr
x_2 & -x_3 & 0
\end{pmatrix}
&
\begin{pmatrix}
0 & \star & \star \cr
0 & \star & \star \cr
0 & \star & \star
\end{pmatrix}
\\ [20pt]
\textrm{Zero row} & 
\textrm{Zero square} & 
\textrm{Antisymmetric} & 
\textrm{Zero column}
\end{array}$ \\[5pt]
\caption{Matrices with vanishing determinant} \label{ta:mat}
\end{table}

For our purposes, the number of variables for the linear forms appearing in the matrices is not especially relevant.  Nevertheless, the {\it{effective}} number of variables is at most~$9$, as we can always replace the initial variables with a basis for the span of the entries of the $3 \times 3$ matrix in question.  As a consequence of Theorem~\ref{t:comp}, if the dimension of the span of the entries of the matrix is at least~$7$, then the determinant cannot vanish identically.  Thus, in the interesting cases, the number of variables could be limited to~$6$.

In a broader context, for $m,n,r$ positive integers, the loci of $m \times n$ matrices with entries in a field and rank at most $r$ are called {\it{determinantal varieties}}, see for instance~\cite{harris}*{Lecture~$9$}.  The extension of the definition of determinantal varieties to the case in which the entries of the matrix are linear forms is especially relevant in the context of rational normal scrolls and minimal varieties: all of this is very classical and the paper~\cite{eiha} is a good introduction to the topic, its history and relevant results.  Eisenbud introduces in~\cite{ei} the property of {\emph{$1$-genericity}} for spaces of matrices: a sort of non-degeneracy for the entries of a matrix of linear forms.  One of his aims is to obtain conditions under which sections of small codimension of certain spaces of matrices inherit the irreducibility property.  His ideas have later been applied, developed and extended, see for instance~\cites{eiret,katz,raiant,raicat,kupula}.  Our work starts with similar questions, but we do not limit the entries of our matrices.  First, we do not impose any genericity condition: we study $3 \times 3$ matrices of linear forms with no restrictions.  Second, we find several irreducible components for our schemes and, in fact, we are able to determine all of them.  Finally, this leads to questions regarding a deeper understanding of the irreducible components of the schemes we introduce and their ideals.

\section{Notation and preliminary results}

Let $k$ be a field and let $n$ be a positive integer.  We denote 
\begin{itemize}
\item
the algebra of polynomials in $x_1,\ldots,x_n$ with coefficients in $k$ by $R = k[x_1,\ldots,x_n]$, 
\item
the field of fractions of $R$ by $K = k(x_1,\ldots,x_n)$.
\end{itemize}
Let $M$ be a matrix with entries in $R$.  The matrix $M$ is {\em{homogeneous}} if the entries of $M$ are homogeneous forms of the {\emph{same}} degree; the matrix $M$ is {\em primitive} if the only polynomials in $R$ dividing all the entries of $M$ are constants; the degree of $M$ is $d = \deg M$ if the maximum of the degrees of the entries of $M$ is $d$.  In the cases of interest below, the entries of $M$ will in fact be homogeneous of the same degree; we state our results for not necessarily homogeneous matrices: the statements in the homogeneous case follow at once observing that if the product of two polynomials is homogeneous and non-zero, then the two polynomials are themselves homogeneous.  A {\em{vector}} is a matrix consisting of a single column; a {\em{row vector}} is a matrix consisting of a single row.

We begin with an easy lemma on matrices of rank~$1$ and entries in~$R$.

\begin{lemma} \label{l:rangouno}
Let $a,b$ be positive integers and let $X$ be an $a \times b$ matrix with entries in $R$.  If the rank of $X$ is~$1$, then there exist two non-zero column vectors ${\bf{u}} \in R^a$ and ${\bf{v}} \in R^b$ such that the identity $X = {\bf{u}} {\bf{v}}^t$ holds.  The vectors ${\bf{u}}$ and ${\bf{v}}$ satisfy $\deg {\bf{u}} + \deg {\bf{v}} = \deg X$ and if the matrix $X$ is homogeneous, then they are also homogeneous.
\end{lemma}

\begin{proof}
Denote by ${\bf{u}}'$ a non-zero column of $X$; let $u \in R$ denote the greatest common divisor of all the entries of the vector ${\bf{u}}'$ and let ${\bf{u}}$ be the primitive column vector in $R^a$ such that ${\bf{u}}'=u{\bf{u}}$.  Since the rank of the matrix $X$ is one, all the columns of $X$ are proportional over $K$ to the vector ${\bf{u}}$; since the entries of the matrix $X$ are contained in $R$ and the vector ${\bf{u}}$ is primitive, all the columns of $X$ are obtained from ${\bf{u}}$ by multiplication by an element of $R$.  Denote by ${\bf{v}} = (v_1 \, \cdots \, v_b)^t \in R^b$ the column vector such that, for $i \in \{1,\ldots,b\}$, the $i$-th entry of ${\bf{v}}$ is the factor $v_i \in R$ with the property that $v_i {\bf{u}}$ is the $i$-th column of $X$.  By construction the identity $X = {\bf{u}} {\bf{v}}^t$ holds.  The assertion about homogeneous matrices is immediate.
\end{proof}

The next result is a direct consequence of Cramer's rule.  Recall that for a square matrix $M$, the adjoint of $M$ is the matrix $\adj(M)$ whose entry in position $(i,j)$ is $(-1)^{i+j}$ times the determinant of the matrix obtained from $M$ by removing the $i$-th column and the $j$-th row.

\begin{lemma} \label{l:cramer}
Let $d,r$ be non-negative integers and let $X$ be an $(r+1) \times (r+1)$~matrix of rank $r$ and degree $d$ with entries in $R$.  There exist two non-zero vectors ${\bf{u}}$ and~${\bf{v}}$ with entries in $R$ such that the identities $X {\bf{u}} = (X^t) {\bf{v}} = 0$ and $\deg {\bf{u}} + \deg {\bf{v}} \leq dr$ hold.  If $X$ is homogeneous, then the vectors ${\bf{u}}$ and~${\bf{v}}$ can be chosen to be homogeneous and satisfy the equality $\deg {\bf{u}} + \deg {\bf{v}} = dr$.
\end{lemma}

\begin{proof}
Let $\adj(X)$ denote the adjoint of the matrix $X$ so that the identities 
\[
X \adj(X) = \adj(X) X = \det (X) \operatorname{Id} = 0 
\]
hold.  The matrix $\adj(X)$ has degree at most $dr$; if $X$ is homogeneous of degree $d$, then $\adj(X)$ is homogeneous of degree $dr$.  Since the rank of $X$ is $r$, it follows that the matrix $\adj(X)$ is non-zero; moreover the columns of $\adj(X)$ lie in the kernel of $X$ and hence are proportional over the fraction field $K$ of $R$, because the kernel of $X$ is one-dimensional.  Thus the rank of the matrix $\adj(X)$ is one and we can apply Lemma~\ref{l:rangouno} to the matrix $\adj(X)$: let ${\bf{u}} \in R^m$ and ${\bf{v}} \in R^n$ be vectors such that the equalities $\adj(X) = {\bf{u}} {\bf{v}}^t$ and $\deg {\bf{u}} + \deg {\bf{v}} = \deg \adj(X) \leq dr$ hold.
Finally, the identity $X {\bf{u}} {\bf{v}}^t = 0$ shows that the vector $X {\bf{u}}$ is the zero vector, and similarly the row vector ${\bf{v}}^t X$ is the zero vector, completing the proof.
\end{proof}

To state the next lemma, we introduce a little notation.  Denote by $\mathbb{P}$ the projective space of dimension $n-1$ whose homogeneous coordinate ring is $R$.  Let $r$ be a non-negative integer and let $\ell = (\ell_1 , \ldots , \ell_r) \in R^r$ be a sequence of $r$ linear forms in $R$.  Denote by $\mathscr{L} \subset \mathbb{P}$ the linear subspace of $\mathbb{P}$ defined by the vanishing of the forms in $\ell$ and by $c$ the codimension of $\mathscr{L}$ in $\mathbb{P}$.  Denote by $\mathscr{I}_\mathscr{L}$ the ideal sheaf of $\mathscr{L}$; the sequence $\ell$ defines a surjective homomorphism of sheaves $\bigl( \mathscr{O}_{\mathbb{P}} (-1) \bigr) ^{\oplus r} \longrightarrow \mathscr{I}_\mathscr{L}$, sending $(f_1 , \ldots , f_r)$ to $\sum f_i \ell_i$.  Twisting by $\mathscr{O}_\mathbb{P} (2)$ and taking global sections, we define the vector space $\mathscr{V}_{r,c}^n$ as the kernel of the homomorphism ${\textrm{H}}^0 \bigl( \mathbb{P} , \mathscr{O}_{\mathbb{P}}(1)^{\oplus r} \bigr) \longrightarrow {\textrm{H}}^0 \bigl( \mathbb{P} , \mathscr{I}_\mathscr{L}(2) \bigr)$.

\begin{lemma} \label{l:dimcou}
The dimension of the vector space $\mathscr{V}_{r,c}^n$ defined above is $(r-c)n + \binom{c}{2}$.
\end{lemma}

\begin{proof}
We will use the diagram 
\begin{eqnarray}
& 0 \nonumber \\
& \uparrow \nonumber \\
0 \longrightarrow \hspace{-12pt} & \mathscr{I}_\mathscr{L} (2) & \hspace{-12pt} \longrightarrow \mathscr{O}_{\mathbb{P}} (2) \longrightarrow \mathscr{O}_\mathscr{L} (2) \longrightarrow 0 \label{eq:dia} \\
& \hphantom{{\scriptstyle{\ell}}} \uparrow {\scriptstyle{\ell}} \nonumber \\
& \mathscr{O}_{\mathbb{P}}(1)^{\oplus r} \nonumber 
\end{eqnarray}
of sheaves on $\mathbb{P}$.  Observe that diagram~\eqref{eq:dia} is exact and stays exact if we take global sections.  By definition, the vector space $\mathscr{V}_{r,c}^n$ is the kernel of the homomorphism 
\[
{\textrm{H}}^0 \bigl( \mathbb{P} , \mathscr{O}_{\mathbb{P}}(1)^{\oplus r} \bigr) \longrightarrow {\textrm{H}}^0 \bigl( \mathbb{P} , \mathscr{I}_\mathscr{L}(2) \bigr) .
\]
Standard computations yield the equalities 
\begin{eqnarray*}
\dim {\textrm{H}}^0 \bigl( \mathbb{P} , \mathscr{I}_\mathscr{L}(2) \bigr) & = & cn - \binom{c}{2} 
\\
\dim {\textrm{H}}^0 \bigl( \mathbb{P} , \mathscr{O}_{\mathbb{P}}(1)^{\oplus r} \bigr) & = & rn
\end{eqnarray*}
and, taking differences, we find $\dim \mathscr{V}_{r,c}^n = (r-c)n + \binom{c}{2}$.
\end{proof}

\begin{lemma} \label{l:trefl}
Let $\ell = (\ell_1 , \ell_2 , \ell_3)$ be a triple of linear forms and let $\mathscr{V}_{\ell}$ be the $k$-vector space of triples of linear forms $(f_1,f_2,f_3)$ such that the identity $f_1 \ell_1 + f_2 \ell_2 + f_3 \ell_3 = 0$ holds.
\begin{enumerate}
\item \label{en:ind}
If the linear forms $\ell_1 , \ell_2 , \ell_3$ are $k$-linearly independent, then $\mathscr{V}_{\ell}$ is spanned by the triples $(0 , \ell_3 , -\ell_2)$, $(-\ell_3 , 0 , \ell_1)$, $(\ell_2 , -\ell_1 , 0)$.
\item \label{en:dep}
If the linear forms $\ell_1 , \ell_2$ are $k$-linearly independent and $\varphi_1 \ell_1 + \varphi_2 \ell_2 + \ell_3 = 0$ is a $k$-linear relation, then $\mathscr{V}_{\ell}$ is spanned by the triples $(\ell_2 , -\ell_1 , 0)$ and $\ell (\varphi_1 , \varphi_2 , 1)$, as $\ell$ varies among all linear forms in $R$.
\end{enumerate}
\end{lemma}

\begin{proof}
We begin with some easy checks: 
\begin{itemize}
\item 
the stated triples satisfy the identity $f_1 \ell_1 + f_2 \ell_2 + f_3 \ell_3 = 0$; 
\item 
if item~\eqref{en:ind} holds, the dimension of the span of the three vectors $(0 , \ell_3 , -\ell_2)$, $(-\ell_3 , 0 , \ell_1)$, $(\ell_2 , -\ell_1 , 0)$ is~$3$; 
\item 
if item~\eqref{en:dep} holds, the dimension of the span of the $n+1$ vectors $(\ell_2 , -\ell_1 , 0)$ and $\ell (\varphi_1 , \varphi_2 , 1)$, as $\ell$ ranges among the $n$ variables of the ring~$R$, is $n+1$.
\end{itemize}
To conclude it suffices to show that the dimension of the vector space $\mathscr{V}_{\ell}$ is~$3$ in case~\eqref{en:ind} and $n+1$ in case~\eqref{en:dep}.

A moment's thought makes it clear that the vector space $\mathscr{V}_{\ell}$ mentioned in this lemma is the same vector space $\mathscr{V}_{r,c}^n$ defined in Lemma~\ref{l:dimcou} in the case $(r,c)=(3,3)$, if the forms $\ell_1,\ell_2,\ell_3$ are independent (item~\eqref{en:ind}), and in the case $(r,c) = (3,2)$, if the forms $\ell_1,\ell_2$ are independent and $\ell_1,\ell_2,\ell_3$ are not (item~\eqref{en:dep}).  Specializing the results of Lemma~\ref{l:dimcou} we finally obtain $\dim \mathscr{V}_{3,3}^n = 3$ and $\dim \mathscr{V}_{3,2}^n = n+1$, as required.
\end{proof}

\section{Main result}

Let $\mathscr{M}_n$ denote the $k$-vector space of $3 \times 3$~matrices whose entries are linear forms in $R$: the dimension of $\mathscr{M}_n$ is $9n$.  We denote the projective space associated to $\mathscr{M}_n$ by $\PM_n$, a projective space of dimension $9n-1$.  By construction, the determinant induces a function $\det \colon \mathscr{M}_n \to R$ whose image is contained in the vector space $R_3$ of forms of degree~$3$ in~$R$.  The dimension of the vector space $R_3$ is $r_3 = \binom{n+2}{3}$, so that the vanishing of the function $\det$ corresponds to the simultaneous vanishing of $r_3$ cubic forms.

\begin{definition}
The space of {\emph{$3 \times 3$~singular matrices of linear forms}} is the subscheme $\mathscr{F}_n$ of $\PM_n$ defined by the vanishing of the function $\det$.
\end{definition}

When the index $n$ is clear from the context, we will omit it; for instance, we may denote the scheme $\mathscr{F}_n$ simply by $\mathscr{F}$.

The space $\mathscr{F}$ is our main interest.  An immediate consequence of the definition of the space $\mathscr{F}$ is that multiplying a $3 \times 3$~matrix of linear forms by an invertible matrix with entries in $k$ on the left or on the right stabilizes $\mathscr{F}$ and its complement.

\begin{definition} \label{d:c}
Two $f \times g$~matrices $X,Y$ with entries in $R$ are $f$-equivalent if there are an invertible $f \times f$~matrix $F$ and an invertible $g \times g$~matrix $G$ both with entries in $k$ such that $Y = FXG$.
\end{definition}

There is a straightforward relationship between $\mathscr{F}_n$ and linear subspaces of determinantal varieties.  Let $\mathbb{P}^8$ denote the $8$-dimensional projective space with coordinates $\{ x_{ij} \}_{i,j \in \{1,2,3\}}$ and let $X$ be the determinantal cubic with equation 
\[
X \colon \hspace{20pt}
\det \begin{pmatrix}
x_{11} & x_{12} & x_{13} \cr
x_{21} & x_{22} & x_{23} \cr
x_{31} & x_{32} & x_{33}
\end{pmatrix} = 0.
\]
A point in the variety $\mathscr{F}_n$ is a linear map of a vector space to the variety $X$.  Equivalently, points of $\mathscr{F}_n$ correspond to (not necessarily injective) parameterizations of linear subspaces contained in $X$.

\begin{theorem} \label{t:comp}
Let $X$ be a $3 \times 3$ matrix of linear forms.  If the determinant $\det(X)$ vanishes, then $X$ is $f$-equivalent to one of the following:
\begin{enumerate}
\item \label{i:rig}
a matrix with a zero row;
\item \label{i:col}
a matrix with a zero column;
\item \label{i:qua}
a matrix with a zero square;
\item \label{i:ant}
an antisymmetric matrix.
\end{enumerate}
\end{theorem}

\begin{proof}
If $X$ is the zero matrix, then there is nothing to prove.  If the rank of $X$ is~$1$, then we can apply Lemma~\ref{l:rangouno} to conclude that either the rows or the columns of $X$ are $k$-proportional and we are in case~\eqref{i:rig} or~\eqref{i:col}.

Thus, we reduce to the case in which the rank of the matrix $X$ is~$2$ and we apply Lemma~\ref{l:cramer}: let $\mathbf{u},\mathbf{v}$ be non-zero column vectors such that $\mathbf{u}$ is in the kernel of $X$, $\mathbf{v}$ is in the kernel of $X^t$ and $\deg \mathbf{u} + \deg \mathbf{v} = 2$.  If one of the vectors $\mathbf{u}$ or $\mathbf{v}$ has degree~$0$, i.e.\ is constant, then the columns or the rows of $X$ are $k$-proportional and we are in case~\eqref{i:rig} or~\eqref{i:col}.  Therefore we further reduce to the case in which the vectors $\mathbf{u}$ and $\mathbf{v}$ have degree~$1$ and are not $R$-proportional to constant vectors: the entries of $\mathbf{u}$ span a $k$-vector space of dimension at least~$2$ of linear forms and similarly for the entries of $\mathbf{v}$.  The rows of the matrix $X$ therefore are triples contained in the vector space $\mathscr{V}_{\mathbf{u}}$ of Lemma~\ref{l:trefl} and, by our reductions, they span a three-dimensional linear subspace of $\mathscr{V}_{\mathbf{u}}$.

If the entries $u_1,u_2,u_3$ of $\mathbf{u}$ are $k$-linearly independent, then after multiplying the matrix $X$ on the left by an invertible matrix with entries in $k$, we may assume that the rows of $X$ are $(0 \; u_3 \; {-u_2})$, $({-u_3} \; 0 \; u_1)$, $(u_2 \; {-u_1} \; 0)$: thus $X$ is $f$-equivalent to the antisymmetric matrix 
$\begin{pmatrix}
0 & u_3 & -u_2 \cr
-u_3 & 0 & u_1 \cr
u_2 & -u_1 & 0
\end{pmatrix}$.

If the entries $u_1,u_2,u_3$ of $\mathbf{u}$ are $k$-linearly dependent and $\varphi_1 u_1 + \varphi_2 u_2 + \varphi_3 u_3 = 0$ is a non-trivial linear relation among them, then after multiplying the matrix $X$ on the left by an invertible matrix with entries in $k$, we may assume that the last two of the rows of $X$ are $(\ell \varphi_1 \; \ell \varphi_2 \; \ell \varphi_3)$ and $(m \varphi_1 \; m \varphi_2 \; m \varphi_3)$, where $\ell,m$ are linear forms.  The matrix $X$ is $f$-equivalent to the matrix 
$\begin{pmatrix}
\star & \star & \star \cr
\ell \varphi_1 & \ell \varphi_2 & \ell \varphi_3 \cr
m \varphi_1 & m \varphi_2 & m \varphi_3
\end{pmatrix}$ and hence, after a sequence of column operations, also to the matrix 
$\begin{pmatrix}
\star & \star & \star \cr
\ell & 0 & 0 \cr
m & 0 & 0
\end{pmatrix}$ with a zero square.
\end{proof}

We define four linear subspaces of $\PM_n$:

\smallskip

\noindent
\begin{tabular}{llll}
$\bullet$
(zero row) &
%Matrices with last row equal to zero: 
$\! R_n = \left\{ \left[ \begin{matrix}
\star & \star & \star \cr
\star & \star & \star \cr
0 & 0 & 0 
\end{matrix} \right] \in \PM_n \right\}$ & 
$\! \dim R_n = 6n-1$; \\[30pt]
$\bullet$
(zero square) &
%Matrices with a zero square: 
$\! S_n = \left\{ \left[ \begin{matrix}
\star & \star & \star \cr
\star & 0 & 0 \cr
\star & 0 & 0
\end{matrix} \right] \in \PM_n \right\}$ & 
$\! \dim S_n = 5n-1$; \\[30pt]
$\bullet$
(antisymmetric) &
%Antisymmetric matrices: 
$\! A_n = \left\{ \left[ \begin{matrix}
0 & \varphi & -\chi \cr
-\varphi & 0 & \psi \cr
\chi & -\psi & 0 
\end{matrix} \right] \in \PM_n \right\}$ & 
$\! \dim A_n = 3n-1$; \\[30pt]
$\bullet$
(zero column) &
%Matrices with last column equal to zero: 
$\! C_n = \left\{ \left[ \begin{matrix}
\star & \star & 0 \cr
\star & \star & 0 \cr
\star & \star & 0 
\end{matrix} \right] \in \PM_n \right\}$ & 
$\! \dim C_n = 6n-1$.
\end{tabular}

\smallskip

Let $G = {\rm GL}_3 \times {\rm GL}_3$ be the group of ordered pairs of invertible $3 \times 3$~matrices.  The group $G$ acts on $M$ by the rule $(f,g) \cdot m = fmg^{t}$ and this action induces an analogous action on $\PM$.  The orbits under the $G$-action are precisely the $f$-equivalence classes of matrices.  Denote 
\begin{itemize}
\item
by $\mathscr{R}_n$ the closure of the $G$-orbit of $R_n$, 
\item
by $\mathscr{S}_n$ the closure of the $G$-orbit of $S_n$, 
\item
by $\mathscr{A}_n$ the closure of the $G$-orbit of $A_n$, 
\item
by $\mathscr{C}_n$ the closure of the $G$-orbit of $C_n$.
\end{itemize}
Each of these schemes is clearly irreducible, being the orbit of a linear subspace under an irreducible group.  As a consequence of Theorem~\ref{t:comp}, we know that each $G$-orbit of a point of $\mathscr{F}$ intersects at least one of the linear subspaces $R$, $S$, $A$, $C$: the set of $k$-rational points of $\mathscr{F}$ is the union of the four subvarieties $\mathscr{R}$, $\mathscr{S}$, $\mathscr{A}$, $\mathscr{C}$.  In Section~\ref{s:stab} we examine these four subvarieties of $\mathscr{F}$.  We shall see that the irreducible components of the scheme $\mathscr{F}$ are the varieties $\mathscr{R}$, $\mathscr{S}$, $\mathscr{A}$, $\mathscr{C}$.

\section{The four components} \label{s:stab}

We analyze here the four components of the scheme $\mathscr{F}_n$ that we determined in Theorem~\ref{t:comp}.  Throughout this section, we assume that the integer $n$ satisfies the inequality $n \geq 2$: the case $n=1$ is essentially the case of $3 \times 3$~matrices over a field.  We begin by determining the stabilizers in $G$ of each of the four linear subspaces $R_n$, $S_n$, $A_n$, $C_n$.

\subsection*{The stabilizers of matrices with a zero row or column}
Let $(f,g)$ in $G$ be a pair preserving the set of matrices $R$ with last row equal to zero.  Observe that we can multiply the second element of the pair, $g$, by any invertible matrix and the pair still preserves the set $R$.  It is then immediate to check that the stabilizers of $R$ and $C$ are the pairs of invertible matrices of the form 
\[
\begin{array}{c@{\hspace{20pt}}c@{\hspace{20pt}}c}
{\textrm{Stabilizer of }} R & & 
{\textrm{Stabilizer of }} C \\[5pt]
\left(
\begin{pmatrix}
\star & \star & \star \cr
\star & \star & \star \cr
0 & 0 & \star 
\end{pmatrix} , 
g 
\right)
& , &
\left(
f , 
\begin{pmatrix}
\star & \star & \star \cr
\star & \star & \star \cr
0 & 0 & \star 
\end{pmatrix} 
\right)
\end{array}
\]
for any $f,g$ in ${\rm GL}_3$.

\subsection*{The stabilizer of matrices with a zero square}
Let $(f,g)$ in $G$ be a pair preserving the set of matrices $S$ with a square of zeros.  For $i,j$ in $\{1,2,3\}$, let $f_{ij}$ and $g_{ij}$ denote the $(i,j)$-th entry of $f$ and of $g$ and let $e_{ij}$ denote the matrix with $(i,j)$-th entry equal to~$1$ and remaining entries equal to~$0$.  The product $f e_{ij} g^t$ is the rank one matrix that is the product of the $i$-th column of $f$ times the $j$-th row of $g^t$.  Thus the condition that the two matrices $f e_{11} g^t$, $f e_{12} g^t$ lie in $S$ translates to the equations 
\begin{eqnarray*}
& f_{21} g_{21} = f_{21} g_{31} = f_{21} g_{22} = f_{21} g_{32} = 0 \\[5pt]
& f_{31} g_{21} = f_{31} g_{31} = f_{31} g_{22} = f_{31} g_{32} = 0 .
\end{eqnarray*}
Since the matrix $g$ is invertible, the square on the four entries $g_{21}$, $g_{31}$, $g_{22}$, $g_{32}$ of $g$ cannot vanish and we deduce that the entries $f_{21}$ and $f_{31}$ of $f$ vanish.  Similarly, the entries $g_{21}$ and $g_{31}$ of $g$ vanish as well and finally the stabilizer of $S$ consists of the matrices of the form 
\[
\left(
\begin{pmatrix}
\star & \star & \star \cr
0 & \star & \star \cr
0 & \star & \star 
\end{pmatrix} , 
\begin{pmatrix}
\star & \star & \star \cr
0 & \star & \star \cr
0 & \star & \star 
\end{pmatrix} 
\right).
\]

\subsection*{The stabilizer of antisymmetric matrices}
Let $(f,g)$ in $G$ be a pair preserving the set of antisymmetric matrices $A$; for $s,t,u$ in $\{0,1\}$ define the antisymmetric matrix $e_{stu}$ by 
\[
e_{stu} = 
\begin{pmatrix}
0 & s & -t \cr
-s & 0 & u \cr
t & -u & 0
\end{pmatrix}.
\]
The matrices $f e_{100}g^t , f e_{010}g^t , f e_{001}g^t$ are antisymmetric; in particular, their diagonal entries vanish.  The implied equations force each row of the matrix $g$ to be proportional to the corresponding row of~$f$.  Imposing now the further constraints arising from the matrices $f e_{100}g^t$, $f e_{010}g^t$, $f e_{001}g^t$ being antisymmetric shows that the matrices $f$ and $g$ themselves are proportional; clearly, such matrices stabilize $A$ and we find that the stabilizer of $A$ consists of matrices of the form 
\[
\hphantom{{\textrm{ and }} \lambda \in k^*}
(f , \lambda f) \quad \quad {\textrm{for }} f \in {\rm GL}_3 {\textrm{ and }} \lambda \in k^*.
\]

Having determined the stabilizers of the linear subspaces $R_n$, $S_n$, $A_n$, $C_n$ we now compute the dimensions of the irreducible components of the space $\mathscr{F}_n$ of $3 \times 3$~singular matrices of linear forms.  Indeed, the dimension of the orbit under $G$ of each linear subspace is the sum of the dimension of the linear subspace plus the codimension of their stabilizer:
\begin{eqnarray*}
\dim (\mathscr{R}_n) & = & 6n + 1 , \\
\dim (\mathscr{S}_n) & = & 5n + 3 , \\
\dim (\mathscr{A}_n) & = & 3n + 7 , \\
\dim (\mathscr{C}_n) & = & 6n + 3 .
\end{eqnarray*}

\subsubsection*{The special case $n=2$}
In the special case in which the number of variables of the ring $R$ equals~$2$, the dimensions of the four subvarieties $\mathscr{R}_2$, $\mathscr{A}_2$, $\mathscr{S}_2$, $\mathscr{C}_2$ coincides and equals~$13$.  An easy check shows that the subvarieties $\mathscr{A}_2$ and $\mathscr{S}_2$ coincide in this case, so that $\mathscr{F}_2$ is the union of three subvarieties: $\mathscr{R}_2$, $\mathscr{A}_2 = \mathscr{S}_2$, $\mathscr{C}_2$.  A quick computation using the computer algebra package {\sc{Magma}} shows that, as subvarieties of the projective space $\PM_2 \simeq \mathbb{P}^{17}$, the degree of $\mathscr{R}_2$ and of $\mathscr{C}_2$ is~$15$ and the degree of $\mathscr{S}_2=\mathscr{A}_2$ is~$51$.

We introduce $9n$ coordinate functions on $\mathscr{M}_n$: for $i,j \in \{1,2,3\}$ and $\ell \in \{1 , \ldots , n\}$, let $a_{ij}^\ell$ denote the coefficient of $x_\ell$ in the $(i,j)$-th entry of a matrix in $\mathscr{M}_n$:
\[
\begin{pmatrix}
a_{11}^1 x_1 + a_{11}^2 x_2 + \cdots + a_{11}^n x_n \;\; & a_{12}^1 x_1 + a_{12}^2 x_2 + \cdots + a_{12}^n x_n \;\; & 
a_{13}^1 x_1 + a_{13}^2 x_2 + \cdots + a_{13}^n x_n \\[5pt]
a_{21}^1 x_1 + a_{21}^2 x_2 + \cdots + a_{21}^n x_n \;\; & a_{22}^1 x_1 + a_{22}^2 x_2 + \cdots + a_{22}^n x_n &  \;\;
a_{23}^1 x_1 + a_{23}^2 x_2 + \cdots + a_{23}^n x_n \\[5pt]
a_{31}^1 x_1 + a_{31}^2 x_2 + \cdots + a_{31}^n x_n \;\; & a_{32}^1 x_1 + a_{32}^2 x_2 + \cdots + a_{32}^n x_n &  \;\;
a_{33}^1 x_1 + a_{33}^2 x_2 + \cdots + a_{33}^n x_n 
\end{pmatrix} .
\]

\subsection*{The orbits of matrices with a vanishing row or column}
Let $M$ be a matrix in $\mathscr{M}_n$.  If $M$ lies in $\mathscr{R}_n$, then it follows that the rows of $M$ are linearly independent.  Making use of the coordinates $\{ a_{ij}^\ell \}$ defined above, we form the $3n \times 3$~matrix 
\[
A = \begin{pmatrix}
a_{11}^1 & a_{11}^2 & \cdots & a_{11}^n \;\; & a_{12}^1 & \cdots & a_{12}^n & \;\; a_{13}^1 & \cdots & a_{13}^n \\[5pt]
a_{21}^1 & a_{21}^2 & \cdots & a_{21}^n \;\; & a_{22}^1 & \cdots & a_{22}^n & \;\; a_{23}^1 & \cdots & a_{23}^n \\[5pt]
a_{31}^1 & a_{31}^2 & \cdots & a_{31}^n \;\; & a_{32}^1 & \cdots & a_{32}^n & \;\; a_{33}^1 & \cdots & a_{33}^n 
\end{pmatrix} .
\]
Therefore, in terms of the associated matrix $A$, the matrix $M$ lies in $\mathscr{R}_n$ if and only if the matrix $A$ has rank at most~$2$.  Hence, the $\binom{3n}{3}$ determinants of the $3 \times 3$~submatrices of the matrix $A$ vanish exactly on the subvariety $\mathscr{R}_n$.  Thus, the variety $\mathscr{R}_n$ is actually the determinantal variety of $3n \times 3$~matrices of rank at most~$2$.

Of course, transposed assertions apply to the variety $\mathscr{C}_n$.

\subsection*{The orbits of antisymmetric matrices or of matrices with a zero square}
We do not know a similar explicit description of the varieties $\mathscr{A}_n$ and $\mathscr{S}_n$.  We checked that the ideals of $\mathscr{A}_n$ and $\mathscr{S}_n$ agree with the ideal of $\mathscr{F}_n$ in all degrees up to 4 in a few examples.

\section*{Concluding questions}

An immediate question that arises from our work is to study further the orbit $\mathscr{S}_n$ of the matrices with a square of zeros and the orbit $\mathscr{A}_n$ of the antisymmetric matrices.

We would also like to present the following personal point of view on the initial question.  Computing a determinant involves usually a fair number of operations.  Moreover, if the calculation shows that the determinant vanishes, I usually try to find a ``justification'' of this fact that is independent of computing the determinant.  Possible such justifications are 
\begin{itemize}
\item
some evident linear combination of the rows or columns of the matrix that vanishes (e.g., one row is the sum of two other rows);
\item
the matrix is of odd order and is antisymmetric;
\item
the matrix has ``lots'' of zeros, (e.g., a $3 \times 3$ matrix with a zero square).
\end{itemize}
These are precisely the irreducible components of the schemes $\mathscr{F}_n$!  As the results in this paper handle the case of matrices of order~$3$, we are naturally lead to wonder what happens next.

\subsection*{Question}
Let $n$ be a positive integer.  What are the irreducible components of the spaces of $n \times n$ matrices of linear forms and identically vanishing determinant?


\begin{bibdiv}
\begin{biblist}

\bib{ei}{article}{
   author={Eisenbud, David},
   title={Linear sections of determinantal varieties},
   journal={Amer. J. Math.},
   volume={110},
   date={1988},
   number={3},
   pages={541--575},
   issn={0002-9327}
}

\bib{eiret}{article}{
   author={Eisenbud, David},
   title={Syzygies, degrees, and choices from a life in mathematics.
   Retiring presidential address},
   journal={Bull. Amer. Math. Soc. (N.S.)},
   volume={44},
   date={2007},
   number={3},
   pages={331--359},
   issn={0273-0979}
}

\bib{eiha}{article}{
   author={Eisenbud, David},
   author={Harris, Joe},
   title={On varieties of minimal degree (a centennial account)},
   conference={
      title={Algebraic geometry, Bowdoin, 1985},
      address={Brunswick, Maine},
      date={1985},
   },
   book={
      series={Proc. Sympos. Pure Math.},
      volume={46},
      publisher={Amer. Math. Soc., Providence, RI},
   },
   date={1987},
   pages={3--13}
}

\bib{harris}{book}{
   author={Harris, Joe},
   title={Algebraic geometry},
   series={Graduate Texts in Mathematics},
   volume={133},
   note={A first course;
   Corrected reprint of the 1992 original},
   publisher={Springer-Verlag, New York},
   date={1995},
   pages={xx+328},
   isbn={0-387-97716-3}
}

\bib{katz}{article}{
   author={Katzman, Mordechai},
   title={On ideals of minors of matrices with indeterminate entries},
   journal={Comm. Algebra},
   volume={36},
   date={2008},
   number={1},
   pages={104--111},
   issn={0092-7872}
}

\bib{kupula}{article}{
   author={Kustin, Andrew R.},
   author={Polini, Claudia},
   author={Ulrich, Bernd},
   title={A matrix of linear forms which is annihilated by a vector of
   indeterminates},
   journal={J. Algebra},
   volume={469},
   date={2017},
   pages={120--187},
   issn={0021-8693}
}

\bib{raicat}{article}{
   author={Raicu, Claudiu},
   title={$3\times3$ minors of catalecticants},
   journal={Math. Res. Lett.},
   volume={20},
   date={2013},
   number={4},
   pages={745--756},
   issn={1073-2780}
}

\bib{raiant}{article}{
   author={Raicu, Claudiu},
   title={Secant varieties of Segre-Veronese varieties},
   journal={Algebra Number Theory},
   volume={6},
   date={2012},
   number={8},
   pages={1817--1868}
}

\end{biblist}
\end{bibdiv}
\end{document}